\documentclass[11pt,leqno]{amsart}
\usepackage[english]{babel}
\usepackage{color}
\usepackage{hyperref}

\newtheorem{te}{Theorem}[section]

\newtheorem{co}{Corollary}[section]
\newtheorem{lm}{Lemma}[section]
\newtheorem{de}{Definition}[section]
\newtheorem{re}{Remark}[section]

\usepackage{color}

\newcommand\R{{\mathbb R}}

\usepackage{color}

\begin{document}

\title[Fibrations Over Singular K3 Surfaces]{Fibrations Over Singular K3 Surfaces and New Solutions to the Hull-Strominger System}

\author{Anna Fino, Gueo Grantcharov, Jose Medel}
\address{Dipartimento di Matematica \lq\lq Giuseppe Peano\rq\rq \\ Universit\`a di Torino\\
Via Carlo Alberto 10\\
10123 Torino\\ Italy\\
and  Department of Mathematics and Statistics Florida International University\\
  Miami Florida, 33199, USA}
 \email{annamaria.fino@unito.it, afino@fiu.edu}
 \address{ Department of Mathematics and Statistics Florida International University\\
  Miami Florida, 33199, USA}
\email{grantchg@fiu.edu, jmede022@fiu.edu}

\maketitle

\begin{abstract} Using fibrations over K3 orbisurfaces we construct new   smooth solutions to the Hull-Strominger system. In particular, we prove  that, for $4 \leq k \leq 22$ and $5 \leq r\leq 22$,  the smooth manifolds  $S^1\times \sharp_k(S^2\times S^3)$ and $\sharp_r  (S^2 \times S^4) \sharp_{r+1} (S^3 \times S^3)$,  have  a complex structure with trivial canonical bundle and admit a solution to the Hull-Strominger  system.

\end{abstract}
\section{Introduction}

The Hull-Strominger system   \cite{St, Hu}   is a system of nonlinear PDEs describing the geometry of compactification of heterotic strings with torsion to 4d Minkowski spacetime, which can be regarded as a generalization of Ricci-flat K\"ahler metrics coupled with Hermitian Yang-Mills equation on non-K\"ahler Calabi-Yau 3-folds.
To describe the system,  let $M$ be a compact complex manifold of complex dimension 3 with holomorphically trivial canonical bundle, so that it admits a nowhere vanishing
holomorphic $(3,0)$-form $\psi$.  Let $V$ be a holomorphic  vector bundle over $M$ with a Hermitian metric $H$  along its fibers and let $\alpha' \in \R$ be a constant, also called  the slope parameter.
The  Hull-Strominger system,   for  the fundamental form   $\omega$  of a Hermitian metric  $g$ on $M$,  is given by:
\begin{eqnarray}
&& \label{SS1}F_H\wedge \omega^2=0; \quad F_H^{2,0}=F_H^{0,2}=0,\\[2pt]
&& \label{SS4}d(\|\psi\|_\omega\,\omega^2)=0,\\[2pt]
&& \label{SS3}i\partial \bar\partial \omega=\frac{\alpha'}{4}\,{\rm tr}\left(R_{\nabla}\wedge R_{\nabla}-F_H\wedge F_H\right),
\end{eqnarray}
where    $F_H$ and $R_{\nabla}$  are respectively  the curvatures of   $H$ and of a metric connection $\nabla$  on the tangent bundle $TM$.

The equations  $\eqref{SS1}$  describe the Hermitian-Yang-Mills equations for the    Chern connection $A_H$  and the  equation   $\eqref{SS4}$  says that $\omega$ is conformally balanced.  The equation $\eqref{SS3}$ is the so-called Bianchi identity or anomaly cancellation equation. Note that in the  equation  \eqref{SS3}  there is an ambiguity in the choice of a metric connection $\nabla$ on $TM$, due to its origins in heterotic string theory  \cite{Hu, St}. Also from  physical perspective one has $\alpha' \geq 0$ with $\alpha'=0$ corresponding to the K\"ahler case, but in mathematical literature the case $\alpha'<0$ is also considered \cite{PPZ17unim}. Different choices of  the connection $\nabla$ and their physical meaning are discussed in \cite{D-S}.

The first solutions of the Hull-Strominger  system on compact non-K\"ahler manifolds  were constructed  by Fu and Yau \cite{FY1,FY2}, taking as $\nabla$   the Chern connection of  $\omega$. The  solutions  are  defined on toric bundles over K3 surfaces.  For their construction  Fu and Yau  used  first   the result by  Goldstein and Prokushkin  \cite{GP}   that for a Ricci-flat base and an appropriate choice of the principal torus fibration, the total space has trivial canonical bundle and has  a balanced metric.  Then they showed that particular solutions of the Hull-Strominger system on some principal torus fibrations on K3 manifolds can be reduced to a complex Monge-Amp\`ere type equation for a scalar function $u$ on the base, and solved by  a continuity method type argument inspired from the techniques of Yau in \cite{Yau} (see also  \cite{PPZ16, PPZ18,PPZ18b}).

Since then, and the work by Li and Yau  \cite{LY2005},  different analytical and geometrical aspects of the Hull-Strominger system have been studied with a relevant  influence to  non-K\"ahler  complex geometry (see  for instance  \cite{ADG, CPY, Fei18,Fernandez16, GGS, GRST,  PPZ18}).
Up to now the biggest set  of solutions is provided by the choice  of $\nabla$ given by the Chern connection  \cite{C1, C2,  FHP,  FHP17,FY14, FIUV,   FGV, FTY09,  OUV17, PPZ16, PPZ18, PPZ17, PPZ17unim}. Examples of solutions of the Hull-Strominger system on non-K\"ahler torus bundles over K3 surfaces with the property that the connection $\nabla$  is Hermitian-Yang-Mills have been constructed in \cite{Fernandez18}.

In  \cite{FGV}  the  Fu-Yau result   have been generalized  to K3 orbisurfaces, providing an  extension to Hermitian $3$-folds foliated by non-singular elliptic curves. Here K3 orbisurface is a simply-connected surface with isolated cyclic singularities and trivial canonical bundle - see Section 2 for the  details.  In particular  the following result has been established: let  $(X,\omega_X)$  be  a compact K3  orbisurface   equipped with two anti-self-dual $(1,1)$-forms $\omega_1$ and $\omega_2$  such that $[\omega_1], [\omega_2]\in H^2_{orb}(X,\mathbb{Z})$ and the total space $X$ of the principal $T^2$  orbifold bundle
$\pi: M \rightarrow  X$ determined by them is smooth.  If  there exists  a stable vector bundle $E$ of degree $0$ over $(X,\omega_X)$ satisfying
\begin{equation}\label{onE}
\alpha'(e_{orb}(X)-(c_2(E)-\frac 1 2 c_1^2(E)))=\frac{1}{4\pi^2}\int_X(\|
\omega_1\|^2+\|\omega_2\|^2)\frac {\omega_X^2}{2}.
\end{equation}
where  $e_{orb}(X)$ denotes the orbifold Euler number, then there is a Hermitian structure $\omega_u$ on the complex $3$-fold   $M$, and a metric $h$ along the fibers of  $E$, such that  $(M, \omega_u,  V= \pi^*E,H=\pi^*(h))$ solves the Hull- Strominger system. Here $\omega_u$ depends on a function $u$ on $X$ satisfying the complex Monge-Amp\'ere equation solved in \cite{FY2}.
Applying this extended result to a special type of   K3 orbisurfaces and  using the  topological classification  in  \cite{GL} for compact simply-connected 6-manifolds with  a free $S^1$-action, in \cite{FGV}   it has been shown that,  for   $13\leq k\leq 22$ and $14\leq r \leq 22$,   the smooth manifolds   $S^1\times \sharp_k(S^2\times S^3)$ and $\sharp_r (S^2\times S^4)\sharp_{r+1}(S^3\times S^3)$  admit a solution to the Hull-Strominger system.
 The cases $k=22$ and $r=22$ respectively correspond to the  solutions of Fu and Yau.   The examples     have  the structure of a principal $S^1$-bundle over  Seifert $S^1$-bundles and for them $\alpha '>0$.
 The   simply connected examples are obtained  starting from a  K3 orbisurface  with isolated $A_1$ singular points and trivial orbifold fundamental group and using  partial resolution of singularities by blow-ups.

 In the present paper  we  construct new solutions  to the Hull-Strominger system  considering  other types of K3 orbisurfaces.
 Our main result is the following

\begin{te}\label{theoremA}
    Let $4\leq k \leq 22$ and $5\leq r \leq 22$. Then the smooth manifolds $S^1 \times \sharp_{k}(S^2\times S^3)$ and  $\sharp_{r}(S^2\times S^4)\sharp_{r+1} (S^3\times S^3)$ admit a complex structure with a trivial canonical bundle with a balanced metric and a solution to the Hull-Strominger system via the Fu-Yau ansatz.
\end{te}

Note that  since a quotient of a manifold under an almost free torus action is an orbifold, we can characterize the underlying simply connected compact smooth manifolds $M$ with almost free $T^2$ action which carry a solution of Hull-Strominger system and also have a transversal Calabi-Yau geometry. Since the base is a K3 orbisurface then they are diffeomorphic to $\sharp_{r}(S^2\times S^4)\sharp_{r+1} (S^3\times S^3)$  for $3\leq r\leq 22$. This follows from \cite{GL} and the topological restriction on the second Betti number $b_2$ arising from the fact that the classes of the holomorphic symplectic form and the K\"ahler form of the base survive under the pull-back to $M$. The upper bound is realized by the smooth K3 surface. We note that the remaining spaces missing from Theorem \ref{theoremA} can not be realized by the methods of Fu-Yau ansatz for the surfaces in \cite{IF}.  Moreover, based on the  results in \cite{Fernandez18}, we note  that the constructions in the
present paper   could  likely  be applied to obtain solution of the Hull-Strominger system with
$\nabla$ Hermitian-Yang-Mills.

\smallskip

 We describe shortly the structure of the paper. In Section \ref{Section2}  we collect the necessary information on singular K3 surfaces, which we call orbisurfaces.  In Section \ref{Section3}  we construct  the smooth principal  $T^2$-orbifold
bundles over  the K3 orbisurfaces.   In Section \ref{Section4} we use  the Serre construction to find the appropriate stable bundles over the K3 orbisurfaces. Finally in   Section \ref{Section5}  we prove  Theorem \ref{theoremA} by using a result in \cite{FGV} and the constructions in Sections \ref{Section3} and \ref{Section4}.

\section{K3 orbisurfaces} \label{Section2}
We recall that a  complex orbifold is a topological space together with a cover of coordinate charts where each element of the cover is homeomorphic to a quotient of an open subset in $\mathbb{C}^n$ containing the origin by a finite group $G_x$ and the transition functions are covered by holomorphic maps. We consider in the paper complex orbifolds with cyclic isotropy groups $G_x$, where $x$ is the image of the origin after under the action. The (complex) dimension of the orbifold  is the number $n$ in $\mathbb{C}^n/G_x$. If $n=2$ the orbifold sometimes is called {\bf orbisurface} and this is the terminology which we'll be using throughout the paper. We are interested in orbisurfaces which arise as a complete intersections in weighted projective spaces and follow \cite{IF} for the notations and definitions not provided here.

\begin{de}\label{definition2} An $A_n$ singularity in a complex orbisurface  $X$ is an isolated  singular point $p\in X$ such that there is a neighborhood $W$ of $p$ with a chart $W\cong \mathbb{C}^2/\mathbb{Z}_{n+1}$, or equivalently, $W\cong V(x^2+y^2+z^{n+1})$ 
\end{de}

Recall that $V(x^2+y^2+z^{n+1})$ is the ideal generated by $x^2+y^2+z^{n+1}$. The main objects in this paper are orbifold analogues of K3 surfaces.

\begin{de}\label{definition1} A {\bf singular K3} $X$ is surface with only isolated singularities $\Sigma=\{p_1,...,p_k\}$, such that $H^1(X,\mathcal{O}_X)=0$ and $\omega_X\cong \mathcal{O}_X$, where $\omega_X:=i_*\omega_{X_{reg}}$, $i:X_{reg}\rightarrow X$ is the inclusion, and $\omega_{X_{reg}}$ is the canonical sheaf of $X_{reg},$ where $X_{reg}=X\setminus \Sigma$. Moreover, if we endow $X$ with an orbifold structure $\mathcal{X}=(X,\mathcal{U})$ we call $\mathcal{X}$ a {\bf K3 orbisurface}.
\end{de}

Note that the dualizing sheaf of $X$ in our case is isomorphic to $i_*\omega_{X\setminus \Sigma}$, so the definition is consistent. Suitable K3 orbisurfaces with at worst finite isolated $A_n$ singularities can be found in \cite{R}  and \cite{IF}. Their orbifold structure is inherited from their respective weighted projective spaces. The examples we consider from \cite{IF} are hypersurfaces of weighted projective spaces, them being $X_{30}\subset \mathbb{P}(5,6,8,11)$, $X_{36}\subset \mathbb{P}(7,8,9,12)$, and $X_{50}\subset \mathbb{P}(7,8,10,25)$, where the $d$ in $X_d$ is the degree of the hypersurface and $a_0,a_1,a_2,a_3$ in $\mathbb{P}(a_0,a_1,a_2,a_3)$ are the weights on the variables of the projective space.

The singularities of $X_{36}$ are described in \cite{IF}.  Here we describe the singularities of  $X_{30}$ and $X_{50}$. In general, a  well-formed $\mathbb{P}(a_0,a_1,a_2,a_3)$ has singularities on each of its affine pieces $x_i\neq 0$, which are isomorphic to $\mathbb{C}^2/\mathbb{Z}_{n+1}$ where $n+1$ is the weight of $x_i$. They have singular edges when the $gcd(a_i,a_j)=d>1$, and the singular edges are isomorphic to $\mathbb{P}(a_i,a_j)\cong \mathbb{P}(\frac{a_i}{d},\frac{a_j}{d})$.

The hypersurface $X_{30}\subset \mathbb{P}(5,6,8,11)$ is given by a general polynomial
$$f=w^6 + x^5 + yz^2 + xy^3 + w^2x^2y+wxyz$$
with $(x,y,z,w)$ having weights $(5,6,8,11)$, where for simplicity we consider all the coefficients in the monomials to be ones. This surface is well-formed and quasismooth (see \cite{IF}), so its only possible singularities are inherited from the ones in $\mathbb{P}(5,6,8,11)$. Since $f$ has the monomials $w^6$ and $x^5$, when we restrict to $w\neq 0$ and $x\neq 0$, $f$ restricts to a polynomial of the form $f(1,x,y,z)=1+...$ and $f(w,1,y,z)=1+...$ thus the singular points $(1:0:0:0),$ and $(0:1:0:0)$ are not in the surface. On the contrary, the polynomial $f$ does not have monomials only on $y$ and $z$, thus the singular points $(0:0:1:0),$ and $(0:0:0:1)$ are in the surface, which are of type $A_7$, and $A_{10}$ respectively. Along the singular edge $\mathbb{P}(6,8)\cong \mathbb{P}(3,4)$ we have  $f=x^4+y^3 $, and with Lemma 9.4 in \cite{IF} we calculate that $X_{30}$ intersects at $\lfloor \frac{12}{3\cdot4}\rfloor=1$ points. Following 5.15 in \cite {IF} this is an $A_1$ singularity since $1=gcd(6,8)-1$   from $\mathbb{P}(6,8)$.

Similarly the hypersurface $X_{50}\subset \mathbb{P}(7,8,10,25)$ is given by the polynomial
$$f=z^2+y^5+w^6x+x^5y+w^2+x^2+y^2+wxyz.$$
By the same reasoning as before, we have the monomials $z^2$ and $y^5$, so $X_{50}$ does not contain the singular points $(0:0:1:0)$ and $(0:0:0:1)$, it does contain the points $(1:0:0:0)$ and $(0:1:0:0)$ with $A_6$ and $A_7$ singularities. Along the edge $\mathbb{P}(8,10)\cong \mathbb{P}(4,5)$, we have $f=y^4+x^5$, so $X_{50}$ intersects the singular edge at $\lfloor \frac{20}{4\cdot5}\rfloor=1$ point. Again from 5.15 in \cite{IF} we conclude that this singular point is of type $gcd(8,10)-1$, hence an $A_1$ singularity. Along the edge $\mathbb{P}(10,25)\cong \mathbb{P}(2,5)$, we have $f=z^2+y^5$ which intersects the singular edge at $\lfloor \frac{10}{2\cdot5}\rfloor=1$ point giving an $A_4$ singularity.

In summary, $X_{30}$ has $A_1,A_7,$ and $A_{10}$ singularities; $X_{36}$ has $A_6,A_7,A_3$ singularities, and $A_2$; $X_{50}$ has $A_6,A_7,A_1,$ and $A_4$ {singularities.

In Theorem \ref{th5.1} we'll need Cartier divisors corresponding to characteristic classes $\omega_1,\omega_2$ used for the construction of the principal  $T^2$-bundle $\pi:M\rightarrow X$. To achieve that we first have to calculate the Picard groups of the K3 orbisurfaces. The following is a  known fact, we provide a short proof for completeness.

\begin{lm} \label{lem2.1}
    Let $X$ be a complex orbisurface  with at worst isolated $A_n$ singularities, and $\pi:\Tilde{X}\rightarrow X$ a blow up at a singular point. Then we have two cases:
for $n>1$ we have
$${\text{Pic}}(\Tilde{X})\cong {\text{Pic}}(X)\oplus \langle \mathcal{L}, \mathcal{L}'\rangle;$$\\
for $n=1$ we have $${\text{Pic}}(\Tilde{X})\cong {\text{Pic}}(X)\oplus \langle \mathcal{G}\rangle,$$
where $\mathcal{L}$ is $(n-1)C$, $\mathcal{L}'$ is $C+C'$ in the case of $n>1$ and $\mathcal{G}=C$ in the case of $n=1$. $C$ and $C'$ are exceptional curves.
\end{lm}

\begin{proof}
    We have an exact sequence
    $$0\rightarrow {\text{Pic}}(X) \underset{\pi^*}{\rightarrow} {\text{Pic}}(\tilde{X})\underset{i_B^*}{\rightarrow} {\text{Pic}}(B)\rightarrow 0$$ where $B:=\pi^{-1}(p)$, the exceptional divisor, and $i_B:B\rightarrow\tilde{X}$ is the inclusion.  We also have a projection map $\varphi:\tilde{X}\rightarrow B$ induced by the inclusion and projection $\tilde{X}\subset X\times \mathbb{P}^2\rightarrow\mathbb{P}^2$ whose image is $B$. The pullback $$\varphi^*:{\text{Pic}}(B)\rightarrow {\text{Pic}}(\tilde{X})$$ splits the exact sequence, thus ${\text{Pic}}(\tilde{X})\cong {\text{Pic}}(X)\oplus {\text{Pic}}(B)$.

    Now, in the case of an $A_1$ singularity, the exceptional curve is a single curve isomorphic to $\mathbb{P}^1$. Indeed, after blowing up the singular point of the variety $V(x^2+y^2-z^2)$ we obtain the exceptional curve $\pi^{-1}(p)\cong V(x^2+y^2-1).$ In the case of $A_n$ with $n>1$ the exceptional curve  is $\pi^{-1}(p)\cong V(x^2+y^2)$, which corresponds  to  two  projective lines $C,C'$ intersecting in one point.  Computing the intersection $(n-1)C\cdot C'=1$ we conclude that the minimal number to make Cartier divisors from the Weil divisors $C$ or $C'$ is $n-1.$ Thus $\mathcal{L}=(n-1)C$, $\mathcal{L}'=C+C'$ serve us as a basis for $Pic(B)$.
\end{proof}

We also need the following:

\begin{lm}
     Let $X$ be a generic K3 hypersurface of codimension 1 on a weighted projective space. Then
$${\text{Pic}}(X)\cong \langle \mathcal{O}_X(n)\rangle\cong \mathbb{Z}$$
for some $n$. Moreover, in the orbifold category we have
$${\text{Pic}}^{orb}(X)=\langle\mathcal{O}_{\mathcal{X}}(1)
\rangle.$$
\end{lm}

The calculation \( {\text{Pic}} (X) \cong \mathbb{Z} \)  for generic K3 hypersurfaces can be found in \cite{Be}, and it follows that \( {\text{Pic}}^{\mathrm{orb}}(X) = \langle \mathcal{O}_{\mathcal{X}}(1) \rangle \). The orbifold Picard group in this case is the group of rational Cartier divisors.

Next we calculate the minimal $n$ for which ${\text{Pic}}(X)\cong \langle \mathcal{O}_X(n)\rangle$ for the surfaces $X_{30},X_{36}$, and $X_{50}$.
In general, our approach is to find a polynomial of degree $d_2$ such that the curve in our surface of degree $d_1$, $X_{d_1,d_2}\subset X_{d_1}$ avoids the singularities of $X_{d_1}$, thus defining a Cartier divisor.

In the case of $X_{30}\subset \mathbb{P}(5,6,8,11)$ we have singularities at the points $(0:0:1:0)$, $(0:0:0:1)$ and one point along the singular edge $\mathbb{P}(6,8)$. We can avoid the singular points by choosing a polynomial which has monomials of the type $y^{11}$ and $z^8$. Thus we choose the polynomial $g=y^{11}+z^8+ \ldots$ of degree $88$ to avoid the singular points. The curve $X_{88,30}$ will intersect the singular edge $\mathbb{P}(6,8)$, but we can choose the coefficients of $g$ such that it avoids the point where $X_{30}$ passes through.
Thus we claim that the sheaf $\mathcal{O}_{X_{30}}(88)$ is the generator of ${\text{Pic}}(X_{30})$.  Indeed, it can't be a lower degree since it would have to divide $88$, but by dividing the degree the corresponding curve would have to go through the singular points. Now, we calculate the self intersection of $H:=\mathcal{O}_{X_{30}}(88)$. From the adjunction formula we have $$\omega_{X_{88,30}}=\mathcal{O}_{X_{88,30}}(88),$$ thus the genus $\overline{g}$ of $X_{88,30}$ is  $$\dim H^0(\omega_{X_{88,30}})= \dim H^0(\mathcal{O}_{X_{88,30}}(88))= 45.$$ We obtain the dimension of global sections of $\mathcal{O}_{X_{88,30}}(88)$ by counting all monomials of degree $88$ of the ring $\mathbb{C}[w,x,y,z]/(f,g)$ given that $w,x,y,z$ have weights $5,6,8,11$ respectively. Finally, we have $${\text{deg}}(\omega_{X_{88,30}})=\omega_{X_{30}}\cdot H + H\cdot H,$$ so $2\overline{g}-2=H^2$, and $H^2=88.$

Using the same procedure for $X_{36}$ and $X_{50}$ we have $${\text{Pic}}(X_{36})=\langle \mathcal{O}_{X_{36}}(56)\rangle,\text{ and } {\text{Pic}}(X_{50})=\langle \mathcal{O}_{X_{50}}(56)\rangle.$$
The self intersections are $$\mathcal{O}_{X_{36}}(56)^2=18,\text{ and }\mathcal{O}_{X_{50}}(56)^2=10.$$

This gives us right away the intersections in the orbifold case, which will be
$$\mathcal{O}_{\mathcal{X}_{30}}(1)^2=\frac{1}{88}, \, \mathcal{O}_{\mathcal{X}_{36}}(1)^2=\frac{9}{1568}, \, \mathcal{O}_{\mathcal{X}_{50}}(1)^2=\frac{5}{1568}.$$
The following lemma follows from Example 18.3.4b in \cite{Ful}.

\begin{lm} (Noether's formula for singular surfaces) Let $X$ be a complex orbisurface with at worst finite $A_n$ singularities say $\{A_{n_1},...,A_{n_j}\}$.
Then

$$\chi(X,\mathcal{O}_X)=\frac{1}{12}\left(K_X^2+\chi_{top}(X)+\sum_{i=1}^j n_i\right),$$
where $\chi(X,\mathcal{O}_X)$ is the Euler characteristic of the sheaf cohomology of $\mathcal{O}_X$ and $\chi_{top}(X)$ is the topological Euler characteristic of $X$.

\end{lm}

As a consequence we get:

\begin{co}\label{corollary1} Let $X$ be a complex K3 orbisurface with at worst finite $A_n$ singularities say $\{A_{n_1},...,A_{n_j}\}$. Then

$$\chi_{top}(X)=24-\sum_{i=1}^j n_i.$$
Moreover, $X$ has the following orbifold Betti numbers:
$$b_0=b_4=1\text{, }b_1=b_3=0\text{, and }b_2=22-\sum_{i=1}^j n_i.$$
\end{co}

\begin{proof}
    Firstly, observe that $\chi(X,\mathcal{O}_X)=2$ by the fact that Dolbeault isomorphism is still true for orbifolds \cite{Ba}. Indeed, $$H^2(X,\mathcal{O}_X)\cong  H^0(X,\omega_X)\cong H^0(X,\mathcal{O}_X)\cong \mathbb{C},$$ and $H^1(X,\mathcal{O}_X)=0$ by definition of K3 orbisurface. So we have $$2=\frac{1}{12}\left(\chi_{top}(X)+\sum_i n_i\right).$$
Now, the Hodge numbers are $h^{i,j}= \dim H^j(X,\Omega^i_X),$ $h^{i,j}=h^{j,i}$, and $b_k=\sum_{i+j=k}h^{i,j}$. Thus, $b_1=h^{0,1}+h^{1,0}=2h^{0,1}=2 \dim H^1(X,\mathcal{O}_X)=0$ and by Poincar\'e duality $b_1=b_3$. Similarly, $b_0=dimH^0(X,\mathcal{O}_X)=1$ and $b_0=b_4$. Finally, $\chi_{top}(X)=2+b_2$, thus $b_2=22-\sum_i n_i.$
\end{proof}

\begin{co}
    Let $X$ be a K3 orbisurface  in Miles Reid's 95 list \cite{R}, let $\pi:\tilde{X}\rightarrow X$ be a chain of blow ups of the $A_n$ singular points, where $k$ is the number of irreducible exceptional curves. Then
    $${\text{Pic}}(\tilde{X})\cong \langle\pi^*H\rangle\oplus\mathbb{Z}^k,$$
    and 
    $${\text{Pic}}^{orb}(\tilde{X})\cong \langle\mathcal{O}_{\tilde{\mathcal{X}}}(1)\rangle\oplus\mathbb{Z}^k$$
\end{co}

\begin{proof}
    This follows from applying Lemma \ref{lem2.1}  as many times as needed. Also, the information from  Corollary \ref{corollary1} would give us the second  orbifold  Betti number. This lemmas also apply for the $Pic^{orb}(\tilde{\mathcal{X}})$ . 
\end{proof}

\begin{re}
    Note that the intersections of the exceptional divisors $\mathcal{L}=(n-1)C$ and $\mathcal{L}'=C+C'$ of the blow up of an $A_n$ singular point with $n>1$ are $\mathcal{L}^2=-n(n-1)$, $\mathcal{L}\cdot \mathcal{L}'=-(n-1)$.
\end{re}

\begin{lm}\label{lm2.4}
    \textcolor{black}{Let $X$ be $X_{36}\subset \mathbb{P}(7,8,9,12)$ or $X_{50}\subset  P(7,8,10,25)$, and let $\pi:\tilde{X}\rightarrow X$ be a chain of blow ups of singular points. Then $\tilde{X}$ has an ample divisor $\mathcal{Q}$ such that $\mathcal{Q}^2=2$, and the set of effective divisors $D$ of $\tilde{X}$ such that $D\cdot \mathcal{Q}\leq 2$ is empty.}
\end{lm}

\begin{proof}
Let $D$ be an effective divisor of $\tilde{X}$. The map $\pi$ induces two homomorphims $\pi^*:{\text{Pic}}(X)\rightarrow {\text{Pic}}(\tilde{X})$ and $\pi_*:Cl(\tilde{X})\rightarrow Cl(X)$, where $Cl(Y)$ is the divisor class group of Weil divisors in $Y$ modulo linear equivalence. This last map extends by tensoring over $\mathbb{Q}$ to be $\pi_*:Cl(\tilde{X})_\mathbb{Q}\rightarrow Cl(X)_\mathbb{Q}$. Let $H$ be the generator of ${\text{Pic}}(X)$, and note that all divisors are $\mathbb{Q}$-Cartier \cite{BG}, so $\mathbb{Q}\cong {\text{Pic}}(X)_\mathbb{Q}\cong Cl(X)_\mathbb{Q}$.
The ample divisor $\mathcal{Q}$ will have the form $$\mathcal{Q}=a_0\pi^*H+ \ldots,$$ where the rest of the summands will be divisors that come from the blow ups and $a_0$ a positive integer. Over Chow rings we have the formula $$\pi_*(\pi^*H\cdot [D])=H\cdot \pi_*D.$$
The divisor $\pi_*D$ is a positive linear combination of elements of $Cl(X)$ or zero, so assume its non zero and let $\bar{H}$ be an arbitrary non-zero summand.
Since $\bar{H}$ is $\mathbb{Q}$-Cartier we have that there are integers $m$ and $k$ such that $m\bar{H}=kH$ with $m,k$ coprime and $m>k.$ Thus we have that $\bar{H}=\frac{k}{m}H$, and $$H\cdot \bar{H}=\frac{k}{m}H^2\in \mathbb{Z}.$$
Since this number is an integer it means that $m$ divides $kH^2$, but $k,m$ are coprime so $m$ divides $H^2$. Here we have two cases, one for $X_{36}$ and one for $X_{50}$ where $H^2$ is $18$ and $10$ respectively. In both cases the smallest $H\cdot \bar{H}$ we can get is when $k=1$ and $m=H^2$, which makes $H\cdot \bar{H}=1,$ thus $$\mathcal{Q}\cdot [D]\geq a_0\pi^*H\cdot D =a_0 H \cdot \pi_*D = a_0 H\cdot \bar{H}=a_0,$$ so as long as $a_0>2$, $\mathcal{Q}\cdot [D]>2$.

Now, if $\pi_*D=0$, then $D$ is a positive linear combination of exceptional curves. Exceptional curves only intersect non-trivially with the curves that arise from the same blow up. Again, we have separate cases. One case is when it is a blow up of an $A_n$ singularity with $n>1$, and a second case is when it is a blow up of an $A_1$ singularity. We start with the first case. From blowing up a singularity we obtain two generators $\mathcal{L},\mathcal{L}'\in {\text{Pic}}(\tilde{X})$ and two curves $C,C'\in Cl(\tilde{X})$, with relations $\mathcal{L}=(n-1)C$ and $\mathcal{L}'=C+C'$. Their intersections are $\mathcal{L}\cdot C=-n,\mathcal{L}\cdot C'=1,\mathcal{L}'\cdot C=-1,\mathcal{L}'\cdot C'=-1,$ so
\begin{align*}
   (a\mathcal{L}+b\mathcal{L}')\cdot C&=-na-b,\\
   (a\mathcal{L}+b\mathcal{L}')\cdot C'&=a-b.
\end{align*}
We must choose $a$ and $b$ such that $-na-b>2$ and $a-b>2$. So $a$ and $b$ as coefficients of $\mathcal{L},\mathcal{L}'$ in $\mathcal{Q}$ must satisfy the former inequalities.
In the case of $A_1$ we only have the $\mathcal{L}'=C$ generator, the exceptional divisor. $\mathcal{L}'\cdot C=-2$, so  $a\mathcal{L}'\cdot C=-2a$, hence we need $-2a>2$. For all the cases we have found sufficient conditions on $\mathcal{Q}$ such that $\mathcal{Q}\cdot D>2$ for any effective divisor. Table  1 contains a list of divisors $\mathcal{Q}$ that satisfy the conditions indexed by Euler characteristic of the surface they are over.
\end{proof}

\section{Construction of $T^2$-bundles over orbisurfaces} \label{Section3}

In this  section  we focus on  $T^2$-bundles  over an orbisurface  $X$. The  general theory of such spaces from the foliations view-point is given in \cite{HS}. In particular, such bundles are determined by two rational divisors on the base orbifold and are  constructed via Seifert $S^1$-bundles.  For the definition of {\bf Seifert $S^1$-bundle} we refer to \cite[Definition 4.7.6]{BG}. Roughly speaking,  Seifert  $S^1$-bundles  are  spaces  with  a locally free $S^1$-action, for which the $S^1$-foliation has an orbifold leaf space.  A  multiple leave is an $S^1$-orbit on which the action is not globally free. 

We recall the following result, which is  Theorem 4.7.3 in \cite{BG} and has been  proven by  Koll\'ar in \cite{Kollar2004}.

\begin{te} \label{theor3.1}
Let $X$ be a normal reduced complex space with at worst quotient singularities and $\Delta=\sum_i (1-\frac{1}{m_i})D_i$
be a $\mathbb{Q}$ divisor (this is the data associated to an orbifold). Then there is a one-to-one correspondence between Seifert ${\mathbb C}^*$-bundles $f:Y\rightarrow (X,\Delta)$ and the  following data:

(i) For each $D_i$ an integer $0\leq b_i< m_i$ relatively prime to $m_i$, and

(ii) a linear equivalence class of Weil divisors $B\in Div(X)$.
\end{te}

In the paper we will   consider the case when $Y$ is smooth and we need  to consider the \lq\lq smoothness part\rq\rq \, of Theorem 4.7.7 in \cite{BG}:

\begin{te}

If $(X,\Delta)$ is a locally cyclic orbifold as in the Theorem above and $f: Y \rightarrow (X,\Delta)$  is an $S^1$-orbibundle whose local uniformizing groups inject into the group $S^1$ of the orbibundle, then $f: Y \rightarrow (X,\Delta)$ is a Seifert $S^1$-bundle and $Y$  is smooth.

\end{te}

For algebraic orbifolds this could be refined (see  Theorem 4.7.8 in \cite{BG} and \cite{K05}) and for an orbifold $(X,\Delta)$ with trivial $H^1_{orb}(X,{\mathbb Z})$ the Seifert $S^1$-bundle  $Y$  is uniquely determined by its first Chern class $c_1(Y/X)\in H^2(X,{\mathbb Q})$, which is defined as $[B] + \sum_i \frac{b_i}{m_i} [D_i]$.

We want to use the examples of K3 orbisurfaces   in the previous section  to construct examples carrying solutions of the Hull-Strominger system. To this aim we can prove  the following result  which already  appeared  in  \cite{KMR}. Here we provide also a more direct proof.

\begin{te}
    Let $\mathcal{X}$ be a complex orbisurface with only $A_n$ isolated singularities, let $\pi:\tilde{\mathcal{X}}\rightarrow\mathcal{X}$  be a chain of blow ups of singular points with at least  one $A_n$ with $n>1$, and denote $\bar{H}:=\mathcal{O}_{\mathcal{X}}(1)$ as well as its pullbacks through the chain of blow ups. Assume that the Seifert $S^1$-bundle given by $c_1=\bar{H}$   is smooth. Then there is an ample divisor $\mathcal{E}\in {\text{Pic}}^{orb}(\tilde{\mathcal{X}})\otimes\mathbb{Q}$, and $\mathcal{D},\mathcal{D}'\in {\text{Pic}}^{orb}(\tilde{\mathcal{X}})$ such that
    \begin{enumerate}
        \item $\mathcal{D},\mathcal{D}'$ are primitive with respect to $\mathcal{E}$, i.e. $\mathcal{D}\cdot\mathcal{E}=0,\mathcal{D}'\cdot\mathcal{E}=0$;
        \item if $K_\mathcal{X}=0$ and $\pi^{orb}_1(X)=1$, then $K_{\tilde{\mathcal{X}}}=0$ and $\pi^{orb}_1(\tilde{\mathcal{X}})=1$;
        \item the Seifert $S^1$-bundles  $f_1:Y_1\rightarrow\tilde{\mathcal{X}}$ with $c_1(Y_1/\tilde{\mathcal{X}})=\mathcal{D}$, and $f_2:Y_2\rightarrow Y_1$ with $c_1(Y_2/Y_1)=\pi^*\mathcal{D}'$ are smooth and $H_1(Y_1,\mathbb{Z})=0$ and $H_1(Y_2,\mathbb{Z})=0.$
    \end{enumerate}

\end{te}

\begin{proof}
    We only need to check the case when the last blow up $\varphi:\tilde{\mathcal{X}}\rightarrow\mathcal{Y}$ in the chain of blow ups $\pi:\tilde{\mathcal{X}}\rightarrow\mathcal{X}$ is of an $A_n$ singularity with $n>1.$  We consider the divisors to be of the form:
    \begin{align*}
    \mathcal{E}&:=\frac{k(an-b)(bn-a)}{(n-1)\bar{\mathcal{E}}\cdot \bar{H}}(\varphi^*\bar{\mathcal{E}})-aC-bC',\\
\mathcal{D}&:=\bar{H}-k(bn-a)C,\\
\mathcal{D}'&:=\bar{H}-k(an-b)C';
\end{align*}
where $n$ is the index in $A_n$, $a=mn+c-1$, $b=c+m$,  with $k,m\in \mathbb{Z}^+$ to be determined later.

We first check that $\mathcal{E}$ is indeed an ample divisor via the Nakai–Moishezon ampleness criterion which says that $\mathcal{E}$ is ample if $\mathcal{E}^2>0$ and $\mathcal{E}\cdot D>0$ for any irreducible curve $D$ in $\tilde{\mathcal{X}}$. Any irreducible curve $D$ will either be the proper preimage of a curve in $\mathcal{Y}$ making $\mathcal{E}\cdot D=q(\varphi^*\bar{\mathcal{E}}\cdot D)=q(\bar{\mathcal{E}}\cdot \varphi_*D)>0$ where $q$ is the rational number in front of $\varphi^*\bar{\mathcal{E}}$ in $\mathcal{E}$, or it will be $C$ or $C'$ making $\mathcal{E}\cdot D=\mathcal{E}\cdot C$. The expression $q(\bar{\mathcal{E}}\cdot \varphi_*D)>0$ is true because we assumed that  $\bar{\mathcal{E}}$ is ample and $q$ is a positive number. So we only have to check for the other cases, and if $\mathcal{E}^2>0$ we get:\\
\begin{align*}
    \mathcal{E}\cdot C&=-aC^2-bC\cdot C'\\
     &=\frac{an-b}{n-1}=\frac{n(mn-1)+c(n-1)-m}{n-1}>0\\
     \mathcal{E}\cdot C'&=-aC\cdot C'-bC'^2\\
      &=\frac{bn-a}{n-1}=\frac{c(n-1)+1}{n-1}>0\\
      \mathcal{E}^2&=\frac{k^2(an-b)^2(bn-a)^2\bar{\mathcal{E}}^2}{(n-1)^2(\bar{\mathcal{E}}\cdot \bar{H})^2}-\left(\frac{a(an-b)+b(bn-a)}{n-1}\right)>0,
\end{align*}
the last inequality is because we can choose $k$ to be big enough.

Now we show that $\mathcal{D}$ and $\mathcal{D}'$ are primitive with respect to $\mathcal{E}$. This follows from:
\begin{align*}
    \mathcal{E}\cdot \mathcal{D}&=\mathcal{E}\cdot \bar{H}-k(bn-a)\mathcal{E}\cdot C\\
    &=\frac{k(an-b)(bn-a)}{(n-1)\bar{\mathcal{E}}\cdot\bar{H}}\bar{\mathcal{E}}\cdot\bar{H}-\frac{k(bn-a)(an-b)}{n-1} =0,\\
    \
    \mathcal{E}\cdot \mathcal{D}'&=\mathcal{E}\cdot\bar{H}-k(an-b)\mathcal{E}\cdot C'\\
    &=\frac{k(an-b)(bn-a)}{n-1}-k(an-b)\frac{bn-a}{n-1} =0.\\
  \end{align*}
The second property,  i.e. $K_{\tilde{\mathcal{X}}}=0$ and $\pi^{orb}_1(\tilde{\mathcal{X}})=1$, follows from the fact that $A_n$ singularities are characterized by the property $K_{\tilde{\mathcal{X}}}\cong \pi^*K_{\mathcal{X}}$ for $\pi:\tilde{\mathcal{X}}\rightarrow \mathcal{X}$ a blow-up at the singularity. By definition of $\pi^{orb}_1$ in \cite{K05} and \cite{Thu}, blowing up can only remove classes from the orbifold fundamental group of the base. Since we have $\pi^{orb}_1(\mathcal{X})=1$, then $\pi^{orb}_1(\tilde{\mathcal{X}})=1$.

The condition of smoothness of $S^1$-Seifert bundles is given in Kollar \cite{K05} and Boyer-Galicki \cite{BG}. We have to check if $\mathcal{D}$ and $\mathcal{D}'$ are generators of the local class groups at all the points of $\Tilde{X}$. For the singular points not involved in the blow up, $\bar{H}$ is a generator of their local class group since $\bar{H}$ defines a smooth $S^1$-Seifert bundle by assumption. $\mathcal{D}$ and $\mathcal{D}'$ restricted to those points are equal to $\bar{H}$. For the remaining singular point (that is when $n>2$ since if $n=2$ the blow-up of that singular point will be smooth in that neighborhood), we just need to check if $bn-a$ and $an-b$ are coprime to $n-1$ since the local class group of the remaining singularity is $\mathbb{Z}/(n-1)\mathbb{Z}$. We control $k$ which is multiplying both $bn-a$ and $an-b$ so we can just make it coprime to $n-1$. Observe that $bn-a=c(n-1)+1$, and $an-b=c(n-1)+(n^2-1)m-n$, so
$$an-b\text{ (mod }n-1)= -n \text{ (mod }n-1)=n-2 \text{ (mod }n-1);$$
and
$$bn-a\text{ (mod }n-1)=c(n-1)+1\text{ (mod }n-1)=1\text{ (mod }n-1).$$
Thus the $S^1$-Seifert bundles are smooth.

Finally, we have to check if $\mathcal{D}$ and $\mathcal{D}'$ define $S^1$-Seifert bundles  over $\tilde{X}$ that are simply connected. The conditions for this are given by Koll\'ar \cite{K05}, and for our case we just have to check if there are elements $\alpha, \beta$ in $H_2(\Tilde{X},\mathbb{Z})$ such that $gcd(\mathcal{D}\cap \alpha, \mathcal{D}\cap \beta)=1$, same for $gcd(\mathcal{D}'\cap \alpha, \mathcal{D}'\cap \beta)=1$, with $\alpha, \beta$ not necessarily the same for $\mathcal{D}$ and $\mathcal{D}'$. 

Consider $H$ and $(n-1)C'$ for $\mathcal{D}$, we have $\mathcal{D}\cdot H=H^2=c$ and
$$-\mathcal{D}\cdot (n-1)C'= (bn-a)C\cdot (n-1)C'=bn-a= c(n-1)+1,$$
fulfilling the requirement. Now consider $H$ and $(n-1)C$ for $\mathcal{D}'$, we have $\mathcal{D}'\cdot H=H^2=c$ and
$$-\mathcal{D}'\cdot (n-1)C= (an-b)C'\cdot (n-1)C=an-b.$$
The expression $$an-b=c(n-1)+(n^2-1)m-n\text{ (mod }c)= (n^2-1)m-n\text{ (mod }c),$$ so we need $(n^2-1)m-n$ to be coprime to $c$. By Dirichlet's theorem on arithmetic progressions, which applies since $gcd(n^2-1,n)=1$, we have that there are infinitely many primes for different $m$ of the form $(n^2-1)m-n$, thus we can always choose one such that $(n^2-1)m-n$ is coprime to $c$.
\end{proof}

\begin{re} Note that,  by the classification in  \cite{GL},   the $T^2$-bundles over  K3 orbisurfaces   $X$  we consider  must be   connected sums of the form as in Theorem 1.1, where $k$ and $r$ correspond respectively to $b_2^{orb} (X) -1$ and $b_2^{orb} (X) -2$. Using the lists in \cite{IF}  and \cite{R} we see that $k$ and $r$ respectively   are at least  $4$ and $5$. 
\end{re}

\section{Existence of Stable Bundle} \label{Section4}

The solutions of Hull-Strominger system involve a stable bundle with  a Hermitian-Einstein  structure.
So in this section we will consider a complex algebraic orbisurface  $X$  and use the Serre construction to find the appropriate stable bundles over it. On one side if we consider $X$ simply as a complex algebraic surface, we have the category of sheaves over $X$, and the abelian category of coherent sheaves over $X$. There are analogues if we give $X$ an orbifold structure.

We begin with the definition of stability in the case of normal singular surfaces. For this we need a good definition of $c_1(\mathcal{F})$ for a coherent sheaf $\mathcal{F}$ as in \cite{Bl}. Observe that  for the divisor class group we have an isomorphism $Cl(X)\cong Cl(X_{reg})$  since $codim(\Sigma)>1$. So, we can define $c_1(\mathcal{F})$ as the image under the isomorphism of $c_1(\mathcal{F}|_{X_{reg}})$, where $c_1(\mathcal{F}|_{X_{reg}})=c_1(det(\mathcal{F}|_{X_{reg}}^{\vee\vee}))$. 

\begin{de}
    Let $\mathcal{F}$ be a coherent orbisheaf over the orbisurface $X$, and let $\mathcal{Q}$ be an ample divisor over $X$. Then the \textit{slope} of $\mathcal{F}$ with respect of $\mathcal{Q}$ is
    $$\mu_\mathcal{Q}(\mathcal{F}):=\frac{c_1(\mathcal{F})\cdot\mathcal{Q}}{rk(\mathcal{F})}.$$
    A locally free sheaf $\mathcal{E}$ is called stable if for any subsheaf $\mathcal{F}\subset \mathcal{E}$ with $rk(\mathcal{F})<rk(\mathcal{E})$ we have
    $$\mu_\mathcal{Q}(\mathcal{F})<\mu_\mathcal{Q}(\mathcal{E}).$$
\end{de}

The following lemma is an extension of the Serre's construction to the case of algebraic orbisurfaces. We follow the proof in \cite{HL}
\begin{te}
        Let $\mathcal{X}=(X,\mathcal{U})$ be an orbifold in the notations of \cite{BG} and let $Z\subset X_{reg}$ be a local complete intersection of codimension $2$, and $L,N\in {\text{Pic}}^{orb}(\mathcal{X})$. Then there is an extension
    $$0\rightarrow L\rightarrow E\rightarrow N\otimes \mathcal{I}_Z\rightarrow 0,$$
    where $E$ is locally free if and only if $(L^{\check{}}\otimes N\otimes K_\mathcal{X},Z)$ satisfies the Caley-Bacharach property.
    \end{te}
\vspace{.2 in}
Recall that the Caley-Bacharach property is the following: Suppose that $Z'\subset Z$ is a subscheme with $\ell(Z')=\ell(Z)-1$ and $s\in {\text{Hom}}(\mathcal{O}_\mathcal{X},L^{\check{}}\otimes N\otimes K_\mathcal{X})$ with $s|_{Z'}=0$, then $s|_Z=0$.

\begin{re} \label{rem4.1}
The notation $\tilde{X}_k$ indicates that $k$ has been added to the second Betti number of the base orbisurface $X$ through blow-ups. Here, $X$ is one of the orbisurfaces $X_{30}, X_{36}, X_{50}$. Specifically, this means that $b_2(\tilde{X}_k) = b_2(X) + k = 4 + k$. Thus, $\tilde{X}_k$ includes additional cohomology classes arising from the exceptional divisors of the blow-ups.
\end{re}

The next two Lemmas provide the existence of stable bundles on the K3 orbisurfaces obtained by blow-ups of $X_{36}$ and $X_{50}$ from Section \ref{Section2}.

\begin{lm} \label{lemma4.1}
    Let $\tilde{X}_k$ be the orbisurface obtained by blowing up singular points of $X_{36},X_{50}$, such that $0\leq k\leq 18$ with $b_2(\tilde{X}_k)=k+4$. Then for any $k>1$ there exists on $\tilde{X}_k$  a stable bundle $E$ of rank $2$ with $c_1(E)=0$ and $c_2(E)=c$ for any $c\geq 5$.
\end{lm}

The case when $k=1$ will be discussed after this lemma. In that case we blow up an $A_1$ singularity, and it becomes impossible to use the following methods to find an ample divisor $Q$ such that $$c_2(E)=2+Q^2<e_{orb}(\tilde{\mathcal{X}})=7-\left (\frac{7}{8}+\frac{10}{11}\right ).$$ We will make it work by doing the constructions on the orbifold category. 

\begin{proof} Let $\pi:\tilde{X}\rightarrow X$ be a chain of blow ups of singular points. Assuming Lemma \ref{lm2.4}  we choose an ample divisor $Q$ such that $Q^2=2$. 

We construct a rank 2  bundle $E'$ with $det(E')=2Q$ and $c_2(E')=\ell$, so the bundle $E=E'\otimes \mathcal{O}_{\tilde{X}}(2Q)$ will be stable with $c_1(E)=0$ and $c_2(E)=c_2(E')-Q^2=\ell-Q^2=\ell-2.$ And we choose a codim 2, $Z\subset \tilde{X}\setminus \Sigma$ such that $\ell(Z)=\ell.$

We can check  the Cayley–Bacharach
condition for the pair $(\mathcal{O}_{\tilde{X}}^{\check{}} \otimes (2Q)\otimes K_{\tilde{X}},Z).$
 Let $\ell_1=h^0(K_{\tilde{X}}\otimes(2Q))=h^0(2Q).$ Since $Q$ is ample by the Kawamata–Viehweg vanishing theorem $\chi(2Q)=\ell_1$. We calculate $$\chi(2Q)=\chi(\mathcal{O}_{\tilde{X}})+\frac{1}{2}(2Q)^2=2+4=6.$$ If we let $\ell>\ell_1$, then by the Cayley–Bacharach condition implies that $E'$ is a locally free sheaf of rank $2$.

    Suppose that $N\subset E'$ was a destabilizing line bundle. Then
    $$\mu_Q(N)\geq \mu_Q(E')=Q^2>0=\mu_Q(\mathcal{O}_{\tilde{X}}),
    $$
  which implies that    $N$  can not be contained in $ \mathcal{O}_{\tilde{X}}.$
As a consequence the composition map
    $$ N\rightarrow E'\rightarrow 2Q\otimes \mathcal{I}_Z $$
     is non zero and it vanishes along $D$ with $Z\subset D$ and $D\cdot Q=\mu_Q(2Q)-\mu_Q(N)\leq Q^2=2.$\\
    By the Lemma \ref{lm2.4} we assumed such $D$ does not exist.

\end{proof}

We note that the lemma is independent of the orbifold structure. 
That is, the vector bundle $E$ (locally free sheaf) is independent of the orbifold structure over $\tilde{X}_1$ with $k>1$. It may construct a stable vector bundle for the case when $k=1$, but in that case any ample divisor $Q$ would have $Q^2>2$, which in turn will make $c_2(E)>5$ and $e_{orb}(\tilde{X}_1)-c_2(E)<0$ thus making the slope parameter $\alpha'<0$.

We note here that we can construct many stable bundles $E$ for which $\alpha'<0$. Existence of solutions for the Hull-Strominger system in this case follows in the same way as  for $\alpha'>0$ in \cite{FY1}. By Theorem 1 in \cite{PPZ17} one even has a solution for which the function $e^u$ can have pre-assigned integral as long as it is large enough. However, as in \cite{FY1} we consider  only the case  $\alpha' > 0$. The following lemma solved this issue.

\begin{lm} \label{lemma4.2}
    Let $X=X_{30}\subset\mathbb{P}(5,6,8,11)$ which has $A_1,A_7,A_{10}$ and let $\pi:\tilde{X}\rightarrow X$ the blow up of the $A_1$ singularity, and let $\tilde{\mathcal{X}}=(\tilde{X},\mathcal{U})$ be the ``canonical" orbifold structure. Then there exist a stable vector orbibundle (called also V-bundle)  $E$ over $\tilde{\mathcal{X}}$ with $c_1^{orb}(E)=0$ and $c_2^{orb}(E)=c-\frac{2}{11}$ for every $c\geq 3$.
\end{lm}

\begin{proof}
    Let $\bar{H}=\mathcal{O}_{\mathcal{X}}(1)$, and let $Q=40\pi^*\bar{H}-3\mathcal{L}'$, where $\mathcal{L}'$ is the exceptional divisor of $\pi$. The divisor $Q$ is ample and $Q^2=\frac{2}{11}$.\\
    Just as before, let $\ell_1=h^0(2Q)=\chi(2Q)$ and
    \begin{align*}
        \chi(2Q)&=\chi_{orb}(2Q)+\mu_{Sing}(2Q)\\
         &= \chi_{orb}(\mathcal{O}_{\tilde{\mathcal{X}}}) + \frac{1}{2}(2Q)^2+ \mu_{Sing} (2Q)\\
          &=\chi(\mathcal{O}_{\tilde{\mathcal{X}}})- \mu_{Sing} (\mathcal{O}_{\tilde{\mathcal{X}}})+2Q^2+ \mu_{Sing} (2Q)\\
          &=2-\left(\frac{10}{11}+\frac{21}{32}\right)+\frac{4}{11}+ \mu_{Sing} (2Q),
    \end{align*}
 where $\mu_{sing}$ is  the correction term in the orbifold Riemann-Roch formula (see    Remark   \ref{rem4.3}).
    
Observe that $\mu_{Sing}(2Q)\leq  \mu_{Sing} (\mathcal{O}_{\tilde{\mathcal{X}}}) = \frac{10}{11}+\frac{21}{32}$, so

\begin{align*}
    \chi(2Q)&\leq 2 - \mu_{Sing}(\mathcal{O}_{\tilde{\mathcal{X}}}) + \frac{4}{11} + \mu_{Sing} (\mathcal{O}_{\tilde{\mathcal{X}}})\\
    &=2+\frac{4}{11}<3=\ell.
\end{align*}
    If we apply Serre's construction to $L=\mathcal{O}_{\tilde{\mathcal{X}}}$, $N = 2Q$, and $Z$ such that $\ell(Z) = 3$, then we obtain a vector orbibundle  $E'$ with $c_1^{orb}(E')=2Q$, and $c_2^{orb}(E')= \ell(Z) = 3$.

    Same as before, assume there is an destabilizing sheaf $N\subset E'$. The map $N\rightarrow E'\rightarrow 2Q\otimes \mathcal{I}_Z$ vanishes at an effective divisor $D$ with $D\cdot Q=2Q\cdot Q - N \cdot Q\leq 2Q\cdot Q = \frac{4}{11}$ but there are no effective divisors with $D\cdot Q\leq \frac{4}{11}$.

    The tensor product of a stable bundle by a line bundle is still stable, so define $E:=E'\otimes(Q^{\vee})$, so $c_1^{orb}(E)=0$, and $c_2^{orb}(E)= \ell(Z) - Q^2= 3 - \frac{2}{11}.$
\end{proof}

\begin{re}
  We check that $Q$ is ample by the Nakai-Moishezon criterion. In the case of $\tilde{\mathcal{X}}$ we just have to check $Q\cdot \bar{H}=\frac{19}{88}$, $Q\cdot\mathcal{L}'=18$, and $Q^2=\frac{2}{11}$. Also, observe that any effective divisor $D$ will have $Q\cdot D\ge Q\cdot \bar{H}=\frac{40}{88}>\frac{4}{11}=2Q^2.$ We obtained the divisor $Q$ the following way. Let $Q=a\pi^*H_0-b\mathcal{L}'$ such that $Q\cdot\bar{H}>2Q^2$, so we found positive integer solutions for $\frac{1}{88}a>\frac{1}{44}a^2-4b^2>0$.
\end{re}

\begin{re} \label{rem4.3}
From \cite{Bl} we get that  the correction term in the orbifold Riemann-Roch formula is given by  $$\mu_{Sing}(\mathcal{F}) : =\underset{x\in Sing}{\sum}\mu_{\mathcal{X},x}(\mathcal{F)}, $$ where 
    $$\mu_{\mathcal{X}, x}(\mathcal{F}) := \frac{1}{\#G_x} \cdot \sum_{g \in G_x - \{ \text{id} \}} \frac{\text{trace}(\rho(g))}{\det(I_n - g)},$$
  and $\#G_x$ denotes the cardinality of the finite group $G_x$.
    Proposition 2.6 in \cite{Bl} establishes a bijection of V-free sheaves or free orbisheaves over a quotient germ $(X,x)$, and representations of the local isotropic group $G_x$ up to isomorphism. In the formula $\rho$ is the representation corresponding to $\mathcal{F}$ and $g$ in the local action of $g$ in the coordinates of the local smoothing of the orbifold.
    For the given surface $\tilde{\mathcal{X}}$ which has $A_7,A_{10}$ singularities, the maximum possible $\mu_{Sing}(\mathcal{F})$ for $\mathcal{F}$ rank $1$ is $\frac{10}{11}+\frac{21}{32}$. This is achieved by computing all possible $\mu_{\mathcal{X}, x}(\mathcal{F})$, which are $\#G$ many.
\end{re}

\begin{re}
Integrating  equation $\eqref{SS3}$ we get 
    $$\alpha' \int_{\tilde{\mathcal{X}}} (\text{tr}(R \wedge R) - \text{tr}(F_H \wedge F_H))=\alpha'(e_{orb}(\tilde{\mathcal{X}})-c_2^{orb}(E)),$$
    and the expression $e_{orb}(\tilde{\mathcal{X}})-c_2^{orb}(E)=7-\left(\frac{10}{11}+\frac{7}{8}\right)-3+\frac{2}{11}=\frac{211}{88}>0.$
\end{re}

\section{Proof of  Theorem \ref{theoremA}}  \label{Section5}

The topology of $T^2$ bundles over complex surfaces is determined in \cite{GGP}. The same results hold without change  of the orbisurfaces we consider.  The method was also used to determine the diffeomorphism type of the spaces $M$ in \cite{FGV}. We start by recalling the following result:

\smallskip

\begin{te}[\cite{FGV}] \label{th5.1}
    Let \( X \) be a compact K3 orbisurface with a Ricci-flat Kähler form \( \omega_X \) and orbifold Euler number \( e_{orb}(X) \). Let \( \omega_1 \) and \( \omega_2 \) be anti-self-dual \((1,1)\)-forms on \( X \) such that \( [\omega_1], [\omega_2] \in H^2_{\text{orb}}(X,\mathbb{Z}) \) and the total space \( M \) of the principal \( T^2 \) orbifold bundle \( \pi: M \to X \) determined by them is smooth. Let \( E \) be a stable vector bundle of degree 0 over \( (X, \omega_X) \) such that
\[
\alpha'(e_{orb}(X) - (c_2(E) - \frac{1}{2} c_1^2(E))) = \frac{1}{4\pi^2} \int_X \left( \|\omega_1\|^2 + \|\omega_2\|^2 \right) \frac{\omega_X^2}{2}.
\]
Then \( M \) has a Hermitian structure \(\omega_u \) and there is a metric \( h \) along the fibers of \( E \) such that \( (M, \omega_u, V = \pi^*E, H = \pi^*(h)) \) solves the Hull-Strominger system with   $\nabla$   being the Chern connection on $M$.
\end{te}
 In Table 3 at the end we list the values $e_{orb}(\tilde{X}_k) - (c_2(E) - \frac{1}{2} c_1^2(E))$ for each k.
 
Now Theorem \ref{theor3.1}, Lemma \ref{lemma4.1}, and Lemma \ref{lemma4.2}  give us sufficient conditions to apply Theorem \ref{th5.1} as a last step to prove Theorem \ref{theoremA}. To use Theorem \ref{th5.1} it remains to calculate
$$e_{orb}(X) - (c_2(E) - \frac{1}{2} c_1^2(E))$$
which will end up being a positive rational number. The case for a stable bundle over $\tilde{X}_1$ with $b_2(\tilde{X}_1)=5$ is explained in Remark  \ref{rem4.1}.

In \cite{BG} and \cite{Bl}, the orbifold Euler number $e_{orb}(X)$ is calculated the following way

$$e_{orb}(X)=e(X)-\sum_i\frac{n_i}{n_i+1},$$
where $n_i$ are the $A_{n_i}$ singularities.

The left hand side of the integral comes from the integral

$$\alpha'\int_X(\text{tr}(R\wedge R)-\text{tr}(F_h\wedge F_h))=\alpha'\left(e_{orb}(X)-\left(\int_X c_2^{orb}(E)-\frac{1}{2}c_1^{orb}(E)^2\right)\right).$$

In our construction of $E$ we get that $c_1^{orb}(E)=0$, so the integral reduces to
$$\alpha'\left(e_{orb}(X)-\int_X c_2^{orb}(E)\right).$$
In order to know the number $\int_X c_2^{orb}(E)$ we need to calculate the orbifold Euler characteristic of the bundle, that is
$$\chi_{orb}(X,E)=2\cdot \chi_{orb}(X,\mathcal{O}_X)-\int_Xc^{orb}_2(E).$$\\

From \cite{Bl} we get the formula
$$
\chi(X,F)= \chi_{orb}(X,F)+ \mu_{Sing}(F)
$$
where $\mu_{Sing} (F):=\sum_{x\in \Sigma}\mu_{X,x}(F),$ for any vector bundle $F$. the map $\mu_{Sing}$ is a group homomorphism, so by the exact sequence defining $E$, that is
$$0\rightarrow \mathcal{Q}^\vee \rightarrow E \rightarrow \mathcal{Q}\otimes \mathcal{I}_Z\rightarrow 0,$$
and the fact that $Z\cap \Sigma=\emptyset$, we have that $\mu_{Sing}(E)=\mu_{Sing}(\mathcal{Q})+\mu_{Sing}(\mathcal{Q}^\vee)=2 \mu_{Sing} (\mathcal{O}_X)$. Thus we have
$$
\chi_{orb}(X,E)=2\cdot \chi_{orb}(X,\mathcal{O}_X)-\int_Xc^{orb}_2(E),$$
$$
\chi(X,E)- \mu_{Sing}(E)=2(\chi(X,\mathcal{O}_X)-\mu_{Sing} (\mathcal{O}_X))-\int_Xc^{orb}_2(E),
$$
$$
\int_Xc^{orb}_2(E)=2\chi(X,\mathcal{O}_X)-\chi(X,E)=c_2(E)=\ell(Z)-\mathcal{Q}^2=\ell(Z)-2.
$$
Here $Z$ is chosen such that $\ell(Z)>\chi(2\mathcal{Q}) = 2+2\mathcal{Q}^2$ say $\ell(Z) = 3+2\mathcal{Q}^2=7.$ Thus $\int_Xc^{orb}_2(E)=5,$ and this holds  for any of our bundles. Note that these cases are independent of the orbifold structure as opposed to the case of Lemma \ref{lemma4.2} and Remark \ref{rem4.3}.

The Table 1 at the end of the paper lists the divisors $\mathcal{Q}$ such that $\mathcal{Q}$ is ample and $\mathcal{Q}^2=2$ in blow ups of $X_{36}\subset \mathbb{P}(7,8,9,12)$ which has $A_6, A_7,A_3,A_2$, and Table 2 lists the ones for $\mathcal{Q}$ such that $\mathcal{Q}^2=2$ in blow ups of $X_{30}\subset\mathbb{P}(5,6,8,11)$ which has $A_1, A_7, A_{10}$.
Using the notations of Remark 4.1 we have that  the Picard group of the surface $\tilde{X}_k$ with $b_2(\tilde{X}_k)=k+4$ will be
$${\text{Pic}}(\tilde{X}_k)=\langle H, \mathcal{L}_1, \mathcal{L}_1', \ldots, \mathcal{L}_{k/2}, \mathcal{L}_{k/2}'\rangle$$
in the case where $k$ is even, and
$${\text{Pic}}(\tilde{X}_k)=\langle H,\mathcal{G}, \mathcal{L}_1, \mathcal{L}_1', \ldots, \mathcal{L}_{(k-1)/2}, \mathcal{L}_{(k-1)/2}'\rangle$$
in the case where $k$ is odd. Here $\mathcal{L}_i$ and $\mathcal{L}'_i$ come from blow-ups of $A_n$ singularities with $n>1$, and $\mathcal{G}$ is the exceptional divisor of the blow-up of an $A_1$ singularity, see Lemma 2.1.

Now we consider the cases in Table 1  for even $k$.  Let the $k$-tuple
$(a_0,...,a_{k-1})$ in $\mathbb{Z}^k_+$ denote the divisor
$$\mathcal{Q}=a_0H-a_1\mathcal{L}_1-a_2\mathcal{L}_1'-...-a_{k-2}\mathcal{L}_{k/2} -a_{k-1} \mathcal{L}_{k/2}',$$
such that $a_{2i}>a_{2i-1}$.
We check that $\mathcal{Q}$ is ample using Nakai-Moichezon criteria. Note that a class $[D]\in Cl(\tilde{X}_k)$ will be represented by $H$, $C_i$, or $C_i'$ for some $1\leq i\leq k/2$ where $\mathcal{L}_i'=C_i+C_i'$ , which is the exceptional divisor of the blow up of an $A_{n_i}$ singularity. Thus we have to check the cases $\mathcal{Q}\cdot [D]=\mathcal{Q}\cdot H=a_0H^2>0,$ $\mathcal{Q}\cdot C_i=-a_{2i-1}\mathcal{L}_i-a_{2i}\mathcal{L}_i'$. For the first case we have 
 \[ \mathcal{Q}\cdot [D] = \begin{cases}
          \mathcal{Q}\cdot H=a_0H^2>0, \\
          \mathcal{Q}\cdot C_i=-a_{2i-1}\mathcal{L}_i\cdot C_i-a_{2i}\mathcal{L}_i'\cdot C_i =a_{2i-1}n_i+a_{2i}>0,  \\
          \mathcal{Q}\cdot C_i'=-a_{2i-1}\mathcal{L}_i\cdot C_i'-a_{2i}\mathcal{L}_i'\cdot C_i' =a_{2i}-a_{2i-1} >0.
       \end{cases}
    \]    
Now, we consider the cases in Table 2 for odd $k>1$. Let the $k$-tuple
$(a_0,...,a_{k-1})$ in $\mathbb{Z}^k_+$ denote the divisor
$$
\mathcal{Q}=a_0H-a_1\mathcal{G}-a_2\mathcal{L}_1-a_3\mathcal{L}_1'-...-a_{k-2}\mathcal{L}_{(k-1)/2} -a_{k-1} \mathcal{L}_{(k-1)/2}'.
$$
Again we check for ampleness, the class $[D]\in Cl(\tilde{X}_k)$ will be represented by $H$, $C_i$, $C_i'$, or $G$, where $H$, $C_i$, and $C_i'$ are the same as before; and $G$ is the class representing the exceptional curve of the blow up of an $A_1$ singularity, and $\mathcal{G}$  the corresponding line bundle. For this case we have
\[ \mathcal{Q}\cdot [D] = \begin{cases}
          \mathcal{Q}\cdot H=a_0H^2>0 \\
          \mathcal{Q}\cdot C_i=-a_{2i}\mathcal{L}_i\cdot C_i-a_{2i+1}\mathcal{L}_i'\cdot C_i=a_{2i}n_i+a_{2i+1}>0  \\
          \mathcal{Q}\cdot C_i'=-a_{2i}\mathcal{L}_i\cdot C_i'-a_{2i+1}\mathcal{L}_i'\cdot C_i'=a_{2i+1}-a_{2i}>0,\\
          \mathcal{Q}\cdot G = -a_1\mathcal{G}\cdot G = 2a_1>0.
       \end{cases}
    \]
This concludes the proof of the main Theorem.  Below are the three tables, to which we referred above in this Section.
\vspace{1 in}

\begin{table}[h] 
    \centering
    \caption{Surfaces obtained by blowing up singularities $A_6, A_7, A_3, A_2$ of $X_{36}\subset \mathbb{P}(7,8,9,12)$}
    \begin{tabular}{|c|c|c|c|c|}
        \hline
        \textbf{Surface} & \textbf{Blow up} & \textbf{Singularities} & $\mathcal{Q}$\\
        \hline
        $\tilde{X}_2$ & $A_3$ & $A_6, A_7, A_1, A_2$& $(14, 11, 28)$\\
        $\tilde{X}_4$ & $A_2$ & $A_6, A_7, A_1$& $(6, 2, 10, 6, 9)$\\
        $\tilde{X}_6$ & $A_6$ & $A_4, A_7, A_1$& $(5, 1, 5, 2, 11, 1, 3)$\\
        $\tilde{X}_8$ & $A_4$ & $A_2, A_7, A_1$ & $(5, 3, 7, 1, 6, 1, 3, 1, 3)$\\
        $\tilde{X}_{10}$ & $A_2$ & $A_7, A_1$& $(5, 1, 6, 3, 8, 1, 3, 1, 3, 1, 3)$ \\
        $\tilde{X}_{12}$ & $A_7$ & $A_5, A_1$ & $(5, 1, 4, 1, 8, 1, 3, 1, 3, 1, 3, 1, 3)$\\
        $\tilde{X}_{14}$ & $A_5$ & $A_3, A_1$ & $(5, 1, 5, 1, 5, 1, 3, 1, 3, 1, 3, 1, 3, 1, 3)$ \\
        $\tilde{X}_{16}$ & $A_3$ & $A_1, A_1$& $(6, 1, 5, 4, 8, 1, 3, 1, 3, 1, 3, 1, 3, 1, 3, 1, 3)$ \\
        \hline
    \end{tabular}
\end{table}
%\vspace{1 in}
\begin{table}[h]
    \centering
    \caption{Surfaces obtained by blowing up singularities $A_6, A_7, A_1, A_4$ of $X_{50}\subset \mathbb{P}(7,8,10,25)$}
    \begin{tabular}{|c|c|c|c|c|}
        \hline
        \textbf{Surface} & \textbf{Blow up} & \textbf{Singularities} & $\mathcal{Q}$\\
        \hline
        $\tilde{X}_3$ &$A_1,A_4$& $A_6, A_7, A_2$& $(5, 3, 2, 7)$\\
        $\tilde{X}_5$  &$A_2$& $A_6, A_7$& $(5, 2, 1, 4, 2, 6)$\\
        $\tilde{X}_7$ &$A_6$& $A_4, A_7$ & $(5, 2, 1, 4, 1, 4, 1, 5)$\\
        $\tilde{X}_{9}$ &$A_4$& $A_2, A_7$& $(6, 2, 1, 4, 2, 6, 1, 5, 1, 3)$ \\
        $\tilde{X}_{11}$ &$A_2$& $A_7$ & $(6, 2, 1, 4, 2, 5, 1, 5, 1, 3, 1, 3)$\\
        $\tilde{X}_{13}$ &$A_7$& $A_5$ & $(8, 2, 1, 4, 1, 12, 1, 3, 1, 3, 1, 3, 1, 3)$ \\
        $\tilde{X}_{15}$ &$A_5$& $A_3$& $(8, 3, 1, 5, 3, 6, 1, 3, 1, 3, 1, 3, 1, 3,1,3)$ \\
        $\tilde{X}_{17}$ &$A_3$& $A_1$& $(7, 2, 1, 5, 1, 4, 1, 3, 1, 3, 1, 3, 1, 3, 1, 3, 1, 3)$ \\
        \hline
    \end{tabular}
\end{table}

\newpage

\begin{table}[h]
    \centering
    \caption{Values of $e_{orb}(X)-c_2^{rb}(E)$ with $\mathcal{Q}^2=2$; $\tilde{X}_{1}$ is the blow up of the $A_1$ singularity of $X_{30}\subset \mathbb{P}(5, 6, 8, 11)$, the rest of the odd entries are blow ups of $X_{50}\subset \mathbb{P}(7,8,10,25)$, and the even entries are blow ups of $X_{36}\subset \mathbb{P}(7,8,9,12)$.} 
    \begin{tabular}{|c|c|c|c|}
        \hline
        \textbf{Surface} & \textbf{Singularities} & $e(X)$ & $e_{orb}(X)-c_2^{orb}(E)$ \\
        \hline
        $\tilde{X}_{1}$ & $A_7, A_{10}$ & $7$ & $211/88$\\
        $\tilde{X}_{2}$ & $A_6, A_7, A_1, A_2$ & $8$ & $17/168$\\
        $\tilde{X}_{3}$ & $A_6, A_7, A_2$& $9$ & $269/168$\\
        $\tilde{X}_{4}$ & $A_6, A_7, A_1$ & $10$ &  $155/56$\\
        $\tilde{X}_{5}$ & $A_6, A_7$ &  $11$ & $239/56$\\
        $\tilde{X}_{6}$ & $A_4, A_7, A_1$ & $12$ & $139/40$\\
        $\tilde{X}_{7}$ & $A_4, A_7$ & $13$ & $253/40$\\
        $\tilde{X}_{8}$ & $A_2, A_7, A_1$ & $14$ &  $57/8$\\
        $\tilde{X}_{9}$ & $A_2, A_7$ & $15$ & $203/24$\\
        $\tilde{X}_{10}$ & $A_7, A_1$&  $16$ & $77/8$ \\
        $\tilde{X}_{11}$ & $ A_{7}$& $17$ & $89/8$\\
        $\tilde{X}_{12}$ & $A_5,A_1$ & $18$ & $35/3$\\
        $\tilde{X}_{13}$ & $A_5$&  $19$ & $79/6$\\
        $\tilde{X}_{14}$ & $A_3,A_1$& $20$ & $55/4$\\
        $\tilde{X}_{15}$ & $A_3$ & $21$ & $61/4$\\
        $\tilde{X}_{16}$ & $A_1,A_1$& $22$ & $16$\\
        $\tilde{X}_{17}$ & $A_1$& $23$ & $35/2$\\
        \hline
    \end{tabular}
\end{table}

\smallskip
\textbf{Acknowledgements}.
The authors would like to thank Tony Pantev and Mirroslav Yotov for useful and helpful discussions and suggestions.
Anna Fino is partially supported by Project PRIN 2022 \lq \lq Geometry and Holomorphic Dynamics”, by GNSAGA (Indam) and by a grant from the Simons Foundation (\#944448).  Gueo Grantcharov is partially supported by a grant from the Simons Foundation (\#853269). We also thank the anonymous reviewers for the useful comments.


\smallskip

\begin{thebibliography}{AV1}

\bibitem{ADG} L.  \'Alvarez-C\'onsul, A. De Arriba de La Hera, M.  Garcia-Fernandez, {\em Vertex algebras from the Hull-Strominger system}, arXiv:2305.06836.


\bibitem{Ba}W. L. Baily, {\em On the Quotient of an Analytic Manifold by a Group of Analytic Homeomorphisms},  Proceedings of the National Academy of Sciences 40, no. 9 (1954), 804-808.

\bibitem{Bl} R. Blache, {\em Chern classes and Hirzebruch-Riemann-Roch theorem for coherent sheaves on complex-projective orbifolds with isolated singularities}, Math. Z. 222 (1996), no. 1, 7–57.

\bibitem{Be} S. M. Belcastro, {\em Picard lattices of families of K3 surfaces},  University of Michigan, 1997.

\bibitem{BG} C. P.  Boyer,  K. Galicki,  {\em Sasakian geometry},  Oxford Mathematical Monographs, Oxford University Press, Oxford, 2008.

\bibitem{C1}  J. Chu, L. Huang, X. Zhu,
 {\em The Fu-Yau equation on compact astheno-K\"ahler manifolds},   Adv. Math. 346 (2019), 908–945.

\bibitem{C2} J. Chu, L. Huang, X. Zhu, {\em The Fu-Yau equation in higher dimensions}, Peking Math. J.   2 (2019), no. 1, 71–97.

\bibitem{CPY}   T. Collins,  S.  Picard,   S.-T. Yau,
{\em Stability of the tangent bundle through conifold transitions}, Comm. Pure Appl. Math. 77 (2024), no. 1, 284--371.

\bibitem{D-S} X. de la Ossa, E.E. Svanes,  {\em Connections, field redefinitions and heterotic supergravity}, J. High Energ. Phys. 2014, no. 10, 123, front matter+54 pp.

\bibitem{Fei18}  T. Fei, {\em Generalized Calabi-Gray Geometry and Heterotic Superstrings}, Proceedings of the International Consortium of Chinese Mathematicians 2017, 261–281. International Press, Boston, MA, 2020.

\bibitem{FHP}  T. Fei, Z.-J. Huang,  S. Picard,  {\em A construction of infinitely many solutions to the Strominger system}, J. Differential Geom.   117 (2021), no. 1,
 23--39.

\bibitem{FHP17} T. Fei, Z.-J. Huang,  S. Picard,  {\em The Anomaly flow over Riemann surfaces}, Int. Math. Res. Not. IMRN 2021, no. 3, 2134--2165.



 \bibitem{FY14} T. Fei , S.-T. Yau, {\em Invariant Solutions to the Strominger System on Complex Lie
Groups and Their Quotients}, Comm. Math. Phys.   338 (2015), no. 3, 1183--1195.

\bibitem{FIUV} M. Fern\'andez, S. Ivanov, L. Ugarte, R. Villacampa, {\em Non-K\"ahler heterotic-string compactifications with non-zero fluxes and constant dilaton},   Commun. Math. Phys.   288
(2009),  677--697.


\bibitem{FGV} A. Fino, G. Grantcharov, L. Vezzoni, {\em Solutions to the Hull–Strominger System with Torus Symmetry},  Comm.  Math.  Phys.  388  (2021), 947--967.

\bibitem{Ful} W. Fulton, {\em Intersection Theory}, Vol. 2 of Ergebnisse der Mathematik und ihrer Grenzgebiete: a series of modern surveys in mathematics. Folge 3, Springer-Verlag, 1984.

\bibitem{FTY09} J.-X. Fu, L.-S. Tseng, S.-T. Yau, {\em Local heterotic torsional models},  Comm. Math.
Phys.  289 (2009) 1151--1169.


\bibitem{FY1} J.-X. Fu, S.-T. Yau, {\em A Monge-Amp\`ere type equation motivated by string theory},
Comm. Anal. Geom. 15, no. 1 (2007),  29--76.

\bibitem{FY2} J.-X. Fu, S.-T. Yau, {\em The theory of superstring with flux on non-K\"ahler manifolds and the complex Monge-Amp\'ere equation}, J. Differential Geom. 78 (2008), no. 3, 369-428.








 \bibitem{Fernandez16}  M. Garcia-Fernandez, {\em Lectures on the Strominger system}, Travaux Math\'ematiques, Special Issue:
School GEOQUANT at the ICMAT, Vol. XXIV (2016) 7--61.

\bibitem{Fernandez18} M. Garcia-Fernandez, {\em T-dual solutions of the Hull-Strominger system on non-K\"ahler threefolds}, J. Reine Angew. Math.   766 (2020), 137--150.

\bibitem{GGS} M. Garcia-Fernandez, R.  Gonzalez Molina, J. Streets, {\em Pluriclosed flow and the Hull-Strominger system}, arXiv: 2408.11674.

\bibitem{GRST}  M. Garcia-Fernandez,  R.  Rubio,  C.  Shahbazi,  C.  Tipler,
{\em Canonical metrics on holomorphic Courant algebroids},
Proc. Lond. Math. Soc. (3)   125 (2022), no. 3, 700--758.

\bibitem{GL}
R. Z. Goldstein, L. Lininger,  {\em A classification of 6-manifolds with free $S^1$
actions}, in \lq \lq Proceedings of the Second Conference on Compact Transformation
Groups"  (Univ. Massachusetts, Amherst, Mass., 1971), Part I, Lecture Notes in
Math.298, Springer, Berlin, (1972) 316--323.



\bibitem{GP}
E. Goldstein, S. Prokushkin, {\em Geometric model for complex non-K\"ahler manifolds with ${\rm SU}(3)$ structure},
Comm. Math. Phys.  251 (2004), no. 1, 65--78.

\bibitem{GGP} D. Grantcharov, G. Grantcharov, Y.S. Poon, {\em Calabi-Yau connections with torsion on toric bundles}, 
J. Differential Geom.   78 (2008), no. 1, 13–32.

\bibitem{HS}  A.  Haefliger, E.  Salem,  {\em Actions of Tori on Orbifolds},    Ann. Global Anal. Geom.
 9 (1991), no. 1, 37--59.

\bibitem{Hu}  C. Hull,   {\em Superstring compactifications with torsion and space-time supersymmetry}, In Turin 1985 Proceedings “Superunification and Extra Dimensions” (1986), 347–375.

\bibitem{HL} D. Huybrechts, M. Lehn, {\em The Geometry of Moduli Spaces of Sheaves}, 2nd ed., Cambridge Univ. Press
(2010).

\bibitem{IF} A. R. Iano-Fletcher, {\em Working with Weighted Complete Intersections, Explicit Birational Geometry of
3-folds,} London Math. Soc. Lecture Notes Ser., vol. {\bf 281}, Cambridge Univ. Press, Cambridge (2000),
101--173.

\bibitem{KMR} S. Konoplev, J. Medel, R. V. Russell, {\em Non-Kähler Calabi-Yau 3-Folds Arising from Singular K3 Surfaces,} Note di Matematica (conditionally accepted)

\bibitem{Kollar2004} J. Koll\'ar, Seifert Gm-bundles,  preprint arXiv: math/0404386.


\bibitem{K05} J. Koll\'ar,   {\em Einstein metrics on five-dimensional Seifert bundles}, J. Geom. Anal.  {\bf 15} (2005), 445--476.

\bibitem{LiYau}
J. Li, S.-T. Yau, {\em Hermitian Yang-Mills connections on non-K\"ahler manifolds}, Mathematical aspects of string theory (S.-T. Yau ed.), 560-573, World Scient. Publ. 1987.

\bibitem{LY2005} J. Li, S.-T. Yau, {\em The existence of supersymmetric string theory with torsion},  J. Differential Geom.  70  (2005) 143--181.

\bibitem{OUV17} A. Otal, L. Ugarte, R. Villacampa, {\em  Invariant solutions to the Strominger system and
the heterotic equations of motion},  Nuclear Phys. B   {\bf 920} (2017), 442--474.


\bibitem{PPZ16}
D.-H. Phong, S. Picard, X.-W. Zhang, {\em Anomaly flows},  Comm. Anal. Geom. 26  (2018), no. 4, 955--1008.

\bibitem{PPZ18} D.-H. Phong, S. Picard,   X.-W.  Zhang, {\em Fu-Yau Hessian equations},  J. Differential Geom.   118 (2021), no. 1, 147--187.

\bibitem{PPZ18b} D.-H. Phong, S. Picard,   X.-W.   Zhang,  {\em New curvature flows in complex geometry}, Surveys in differential geometry 2017. Celebrating the 50th anniversary of the Journal of Differential Geometry, 331--364. Surv. Differ. Geom., 22
International Press, Somerville, MA, 2018

\bibitem{PPZ17} D.-H. Phong, S. Picard,   X.-W.  Zhang, {\em The Fu-Yau equation with negative slope parameter},  Invent. Math.  209 (2017),
541--576.

\bibitem{PPZ17unim} D.-H. Phong, S. Picard, X.-W. Zhang, {\em The Anomaly flow on unimodular Lie groups}, Contemp. Math. 735
American Mathematical Society, [Providence], RI, 2019, 217--237.

\bibitem{Pi} S. Picard, P. Wu,  {\em Balanced and Aeppli Parameters for the Heterotic Moduli}, {\tt arXiv:2401.05331}

\bibitem{R} M. Reid, {\em The Du Val Singularities $A_n$, $D_n$, $E_6$, $E_7$, $E_8$}. Lecture Notes, \url{https://homepages.warwick.ac.uk/~masda/surf/more/DuVal.pdf } (2012).



\bibitem{St}  A. Strominger,  {\em Superstrings with torsion}, Nuclear Phys. B  274 (2) (1986), 253--284.

\bibitem{Thu} W. Thurston, {\em The geometry and topology of 3–manifolds}, Princeton University, Mimeographed Notes, 1978.

\bibitem{Yau} S.-T. Yau,  {\em  On the Ricci curvature of a compact K\"ahler manifold and the complex Monge-Amp\`ere equation. I}., Comm. Pure Appl. Math.  {\bf 31} (1978), no. 3, 339--411.

\end{thebibliography}
\end{document}